\documentclass[a4paper,11pt]{article}
\usepackage{amsthm}

\usepackage{epsfig}
\usepackage{amsfonts}
\usepackage{amsmath}
\usepackage{amssymb}
\usepackage{hyperref}
\PassOptionsToPackage{unicode}{hyperref}

\usepackage[ english]{babel}
\usepackage{amsmath}
\usepackage{amsfonts}
\usepackage{enumerate}
\usepackage{amssymb}
\usepackage[usenames]{color}
\usepackage{ mathrsfs }

\newtheorem{theo}{Theorem}[section]
\newtheorem{lemm}[theo]{Lemma}
\newtheorem{coro}[theo]{Corollary}

\newtheorem{rema}[theo]{Remark}
\newtheorem{defi}{Definition}

\numberwithin{equation}{section}
\newcommand{\R}{\mathbb{R}}

\newcommand{\eqdef}{\stackrel{{\rm{def}}}{=}}

\usepackage[a4paper]{geometry}
\usepackage{graphicx}
\usepackage{graphics}
\usepackage{epsfig}
\usepackage{amsmath}
\usepackage{amsfonts}
\usepackage{psfrag}

\begin{document}

\begin{center} \large{\bf Existence and regularity of weak solutions for singular elliptic problems} 
\end{center}
%


\vspace{ -1\baselineskip}

{\small
\begingroup\small
\begin{center}
{\sc B. Bougherara}\\
 LMAP (UMR CNRS 5142) Bat. IPRA,
   Avenue de l'Universit\'e \\
   F-64013 Pau, France\\
  {\it e-mail: brahim.bougherara@univ-pau.fr}\\
   [10pt]
{\sc J. Giacomoni}\\
   LMAP (UMR CNRS 5142) Bat. IPRA,
   Avenue de l'Universit\'e \\
   F-64013 Pau, France
\\
 {\it e-mail: jacques.giacomoni@univ-pau.fr}\\
[10pt]
 {\sc J. Hern\'andez}\\
Departemento de Matem\'aticas, Universidad Aut\'onoma de Madrid, 28049 Madrid, Spain\\
{\it e-mail: jesus.hernandez@uam.es}\\
[10pt]

\end{center}
\endgroup
}

\begin{abstract}
In the present paper we investigate the following semilinear singular elliptic  problem:
\begin{equation*}
(\rm P)\qquad 
\left  \{ \begin{array}{l}
-\Delta u = \dfrac{p(x)}{u^{\alpha}}\quad \text{ in } \Omega \\
u = 0\  \text{ on } \Omega,\ u>0 \text{ on } \Omega,
\end{array}
\right .
\end{equation*}
where $\Omega$ is a regular bounded domain of $\mathbb R^{N}$, $\alpha\in\mathbb R$, $p\in C(\Omega)$ which behaves as $d(x)^{-\beta}$ as $x\to\partial\Omega$ with $d$ the distance function up to the boundary and $0\leq \beta <2$. We discuss below the existence, the uniqueness and the stability of the weak solution $u$ of the problem (P). We also prove accurate estimates  on the gradient of the solution near the boundary $\partial \Omega$. Consequently, we can prove that the solution belongs to $W^{1,q§}_0(\Omega)$ for $1<q<\bar{q}_{\alpha,\beta}\eqdef\frac{1+\alpha}{\alpha+\beta-1}$ optimal if $\alpha+\beta>1$. 
\end{abstract}

\section{Introduction}
In this paper, we deal with the following quasilinear elliptic problem  $({\rm P})$:
\begin{equation*}
(\rm P)\qquad 
\left  \{ \begin{array}{l}
-\Delta u = \dfrac{p(x)}{u^{\alpha}}\quad \text{ in } \Omega \\
u = 0\  \text{ on } \Omega,\ u>0 \text{ on } \Omega,
\end{array}
\right .
\end{equation*}
where $\Omega$ is an open bounded domain with smooth boundary in $\mathbb R^{N}$, $0<\alpha $ and $p$ is a nonnegative function. 

Nonlinear elliptic singular boundary value problems have been  studied during the last forty years in what concerns existence, uniqueness (or multiplicity) and regularity of positive solutions.

The first relevant existence results for a class of problems including the model case $({\rm P})$ with $p$ smooth and $\alpha>0$, were obtained in two important papers by [{\sc Crandall-Rabinowitz-Tartar}] \cite{CrRaTa} and [{\sc Stuart}] \cite{St}. Actually both papers deal with much more general problems  regarding the differential operator and the nonlinear terms. They prove the existence of classical solutions in the space $C^2(\Omega)\cap C(\overline{\Omega})$ by using some kind of approximation  procedure: in \cite{CrRaTa}, the nonlinearity in $({\rm P})$ is replaced by the regular term $\frac{p(x)}{(u+\varepsilon)^\alpha}$ with $\varepsilon>0$ and the authors then show that the approximate problem has a unique solution $u_\varepsilon$ and that $\{u_\varepsilon\}_{\varepsilon>0}$ tends to a smooth function $u^*\in C^2(\Omega)\cap C(\overline{\Omega})$ as $\varepsilon\to 0^+$ which satisfies $({\rm P})$ in the classical sense. A different approximation procedure is used in \cite{St}. These results were  extended in different ways by many authors, we can mention the papers [{\sc Hernandez-Mancebo-Vega}] \cite{HeMaVe}, \cite{HeMaVe2}, the surveys [{\sc Hernandez-Mancebo}] \cite{HeMa} and [{\sc Radulescu}] \cite{Ra} and the book [{\sc Gerghu-Raduslescu}] \cite{GhRa} and the corresponding references. We point out that the existence results in \cite{HeMaVe} and \cite{HeMaVe2} are obtained by applying the method of sub and supersolutions without requiring some approximation argument.

The regularity of solutions was also studied in these papers and the main regularity results were stated and proved in [{\sc Gui-Hua Lin}] \cite{GuLi}.  For Problem $({\rm P})$ with $p\equiv\, 1$, the authors obtain that the solution $u$ verifies
\begin{itemize}
\item[(i)] If $0<\alpha<1$, $u\in C^{1,1-\alpha}(\overline{\Omega})$.
\item[(ii)] If $\alpha>1$, $u\in C^{\frac{2}{1+\alpha}}(\overline{\Omega})$.
\item[(iiii)] If $\alpha=1$, $u\in C^\beta(\overline{\Omega})$ for any $\beta\in (0,1)$.
\end{itemize}
Concerning weak solutions in the usual Sobolev spaces, [{\sc Lazer-McKenna}] \cite{LaKe} prove that the classical solution belongs to $H^1_0(\Omega)$ if and only if $0\leq \alpha<3$. This result was generalized later for $p(x)=d(x)^\beta$ with $d(x)\eqdef d(x,\partial\Omega)$ with the  restrictions $\beta>-2$ in [{\sc Zhang-Cheng}] \cite{ZhCh} and with $0<\alpha-2\beta<3$ in [{\sc Diaz-Hernandez-Rakotoson}] \cite{DiHeRa}. Very weak solutions in the sense given in [{\sc Brezis-Cazenave-Martel-Ramiandrosoa}] \cite{BrCaMaRa} using the results for linear equations in [{\sc Diaz-Rakotoson}] \cite{DiRa} are studied in \cite{DiHeRa}.
In the present paper, we give direct and very simple proofs  avoiding the heavy and deep machinery of the classical linear theory (Schauder theory and $L^p$- theory used in \cite{CrRaTa} and \cite{St}) in order to prove existence results for solutions between ordered sub and supersolutions. We do not use any approximation argument. Our main tools are the Hardy-Sobolev inequality in its simplest form, Lax-Milgram Theorem and  a compactness argument in weighted spaces framework from [{\sc Bertsch-Rostamian}] \cite{BeRo}.

\section{Existence for the case $0<\alpha<1$}
We study the existence of positive weak solutions to the nonlinear singular problem:
\begin{eqnarray*}
({\rm P}_0)\left\{\begin{array}{ll}
&-\Delta u=\frac{1}{u^\alpha}\quad\mbox{in }\Omega\\
&u=0\quad\mbox{on }\partial\Omega
\end{array}
\right.
\end{eqnarray*}
where $\Omega$ is a smooth bounded domain in $\R^N$ and $0<\alpha<1$.

The problem $({\rm P}_0)$ is reduced to an equivalent fixed point problem which is studied by using a method of sub and supersolutions giving rise to monotone sequences converging to fixed points which are actually minimal and maximal solutions (which may coincide) in the interval between the ordered sub and supersolutions. In our case the choice of the functional space where to work is given by the boundary behavior of the purported solutions we suspect.
\begin{defi}\label{def2.1}
We say that $u_0$ (resp. $u^0$) is a subsolution (resp. a supersolution) of $({\rm P}_0)$ if $u_0$, $u^0$ belong to $H^1_0(\Omega)\cap L^\infty(\Omega)$ and if 
\begin{eqnarray}\label{sub-sup}
\int_{\Omega}\nabla u_0\nabla v-\int_{\Omega}(u_0^{-\alpha})v\leq 0\leq\int_{\Omega}\nabla u^0\nabla v-\int_{\Omega}(u^0)^{-\alpha}v\quad\mbox{for any }v\in H^1_0(\Omega),\; v\geq 0.
\end{eqnarray} 
\end{defi}
The main existence theorem we shall prove is the following:
\begin{theo}\label{theo2.1}
Assume that there exists a subsolution $u_0$ (resp. a supersolution $u^0$) such that $u_0\leq u^0$ and that there exist constants $c_1$, $c_2$ satisfying :
\begin{eqnarray*}
0<c_1 d(x)\leq u_0(x)\leq u^0(x)\leq c_2 d(x)\quad\mbox{in }\Omega.
\end{eqnarray*}
Then, there exists a minimal solution $\underline{u}$ (resp. a maximal solution $\overline{u}$) such that
\begin{equation*}
u_0\leq \underline{u}\leq\overline{u}\leq u^0.
\end{equation*}
\end{theo}
In order to prove this theorem we define for the weight $b(x)\eqdef \frac{1}{d^{1+\alpha}(x)}$ the subset
\begin{equation*}
K\eqdef [u_0,u^0]=\left\{u\in L^2(\Omega,b)\, :\, u_0\leq u\leq u^0\right\}
\end{equation*}
where $L^2(\Omega,b)$ is the usual weighted Lebesgue space with weight $b(x)$. Notice that $K$ is convex, closed and bounded.

We reduce the original problem $({\rm P}_0)$ to an equivalent problem for a nonlinear operator associated to the solution operator of problem $({\rm P}_0)$. A first auxilary result is the following:
\begin{lemm}\label{lemme2.1}
There exists a positive constant $M>0$ such that the mapping $F\,:\, K\to H^{-1}(\Omega)$ defined by $F(w)=\frac{1}{w^\alpha}+M\frac{w}{d(x)^{1+\alpha}}$ for $M>0$ large enough is well-defined, continuous and monotone.
\end{lemm}
\proof Let $z\in H^1_0(\Omega)$. By using the Hardy-Sobolev inequality and the fact that $w\in K$, we get for the first term of $F(w)$ :
\begin{eqnarray*}
\displaystyle\left\vert <\frac{1}{w^\alpha},z>\right\vert=\left\vert\int_\Omega\frac{z}{w^\alpha}{\rm d}x\right\vert\leq c\int_{\Omega}\left\vert\frac{z}{w^\alpha}\right\vert d^{1-\alpha}{\rm d}x\leq c\left\Vert\frac{z}{d}\right\Vert_{L^2(\Omega)}\leq c\Vert z\Vert
\end{eqnarray*}
where $c$ denotes (as all along the paper) different positive constants which are independent of the functions involved. In the same vein, we denote by $\Vert u\Vert$ the norm $\left(\int_\Omega\vert\nabla u\vert^2{\rm d}x\right)^{\frac{1}{2}}$ in the Sobolev space $H^1_0(\Omega)$.

For the second term of $F(w)$ we have for any $z\in H^1_0(\Omega)$,
\begin{eqnarray*}
\displaystyle\left\vert<\frac{w}{d^{1+\alpha}},z>\right\vert=\left\vert\int_{\Omega}\frac{wz}{d^{1+\alpha}}{\rm d}x\right\vert\leq \int_{\Omega}\left\vert\frac{z}{d}\right\vert.\left\vert \frac{w}{d^\alpha}\right\vert{\rm d}x\leq c\Vert z\Vert
\end{eqnarray*}
where the constant $c>0$ is given by
\begin{eqnarray*}
\displaystyle\left\Vert\frac{w}{d^\alpha}\right\Vert_{L^2(\Omega)}=\left\vert\int_\Omega\frac{w^2}{d^{2\alpha}}{\rm d}x\right\vert^{\frac{1}{2}}=\left(\int_\Omega\frac{w^2}{d^{1+\alpha}}d^{1-\alpha}{\rm d}x\right)^{\frac{1}{2}}\leq c\Vert w\Vert_{L^2(\Omega,b)}.
\end{eqnarray*}
The existence of the constant $M>0$ such that $F$ is monotone increasing can be obtained by reasoning as in \cite{HeMaVe}. Notice that we only work in the bounded interval $[0,\max u^0]$. Next we prove the continuity of $F$. For the first term, if we assume that $w_n\to w$ in $L^2(\Omega,b)$, we should prove that 
\begin{eqnarray*}
\displaystyle\left\Vert \frac{1}{w_n^\alpha}-\frac{1}{w^\alpha}\right\Vert_{H^{-1}(\Omega)}\to 0\quad \mbox{ as } n\to\infty.
\end{eqnarray*}
We have
\begin{eqnarray*}
\displaystyle\left\vert\int_\Omega\left(\frac{1}{w_n^\alpha}-\frac{1}{w^\alpha}\right)z{\rm d}x\right\vert=\left\vert\int_\Omega\frac{w^\alpha-w_n^\alpha}{w_n^\alpha w^\alpha}\left(\frac{z}{d}\right)d{\rm d}x\right\vert\leq c'_n\left\Vert\frac{z}{d}\right\Vert_{L^2(\Omega)}\leq c'_n\Vert z\Vert
\end{eqnarray*}
where now we have by using the mean value theorem and the definition of $K$ that
\begin{eqnarray*}
&\displaystyle c_n'=\left\Vert \frac{d(w^\alpha-w_n^\alpha)}{w_n^\alpha w^\alpha}\right\Vert_{L^2(\Omega)}=\left(\int_\Omega\frac{\alpha^2 w(\theta)^{2(\alpha-1)}\vert w-w_n\vert^2 d^2}{\vert w_n\vert^{2\alpha}\vert w\vert^{2\alpha}}{\rm d}x\right)^{\frac{1}{2}}\leq\\
&\displaystyle c\left(\int_\Omega\frac{\vert w_n-w\vert^2 d^{2(\alpha-1)}d^2}{d^{4\alpha}}{\rm  d}x\right)^{\frac{1}{2}}\leq c\left(\int_\Omega\frac{\vert w-w_n\vert^2}{d^{2\alpha}}{\rm d}x\right)^{\frac{1}{2}}\leq c\left(\int_\Omega\frac{\vert w-w_n\vert^2}{d^{1+\alpha}}{\rm d}x\right)^{\frac{1}{2}}\leq\\
&\displaystyle c\Vert w-w_n\Vert_{L^2(\Omega,b)}
\end{eqnarray*}
which converges to $0$ as $n\to\infty$ (here $\theta$ denotes the intermediate point in the segment). For the second term in $F$, we have for any $z\in H^1_0(\Omega)$:
\begin{eqnarray*}
\displaystyle\left\vert <\frac{w-w_n}{d^{1+\alpha}},z>\right\vert\leq\int_\Omega\frac{\vert w-w_n\vert\vert z\vert}{d^{1+\alpha}}{\rm d}x=\int_\Omega\frac{\vert w-w_n\vert}{d^\alpha}\left\vert\frac{z}{d}\right\vert{\rm d}x.
\end{eqnarray*}
We have now
\begin{eqnarray*}
\displaystyle\int_\Omega\frac{\vert w-w_n\vert^2}{d^{2\alpha}}{\rm d}x=\int_\Omega\frac{\vert w-w_n\vert^2}{d^{1+\alpha}}d^{1-\alpha}{\rm d}x\leq c\Vert w-w_n\Vert^2_{L^2(\Omega,b)}
\end{eqnarray*}
from where we obtain
\begin{eqnarray*}
\displaystyle\left\vert<\frac{w-w_n}{d^{1+\alpha}},z>\right\vert\leq c\Vert w-w_n\Vert_{L^2(\Omega,b)}\Vert z\Vert
\end{eqnarray*}
giving the result.\qed

Problem $({\rm P}_0)$ is obviously equivalent to the nonlinear problem
\begin{eqnarray}
\label{eq2.3}
\left\{\begin{array}{ll}
&-\Delta u+\frac{Mu}{d(x)^{1+\alpha}}=\frac{1}{u^\alpha}+\frac{Mu}{d(x)^{1+\alpha}}\quad\mbox{in }\Omega,\\
&u=0\quad\mbox{on }\partial\Omega.
\end{array}\right.
\end{eqnarray}
Now we "factorize" conveniently the solution operator to \eqref{eq2.3}. With this aim, we prove first the following result
\begin{lemm}\label{lemme2.2}
If $0<\alpha<1$, for any $h\in H^{-1}(\Omega)$, there exists a unique solution $z\in H^1_0(\Omega)$ to the linear  problem
\begin{eqnarray}\label{eq2.4}
\left\{\begin{array}{ll}
&-\Delta z+\frac{M z}{d(x)^{1+\alpha}}=h\quad\mbox{in }\Omega,\\
&z=0\quad\mbox{on }\partial\Omega.
\end{array}\right.
\end{eqnarray}
Moreover, if $h\geq 0$ (in the sense that $\displaystyle<h,z>_{H^{-1},H^1_0}\geq 0$ for any $z\in H^1_0(\Omega)$ satisfying $z\geq 0$ a.e. in $\Omega$), then $z\geq 0$.
\end{lemm}
\proof
We apply Lax-Milgram theorem. Indeed, the associated bilinear form
\begin{eqnarray*}
\displaystyle a(u,v)=\int_\Omega\nabla u .\nabla v{\rm d}x +M\int_{\Omega}\frac{ u v}{d(x)^{1+\alpha}}{\rm d}x
\end{eqnarray*}
is well-defined, continuous and coercive in $H^1_0(\Omega)$. Using again Hardy-Sobolev inequality we get
\begin{eqnarray*}
\displaystyle\left\vert\int_\Omega\frac{uv}{d^{1+\alpha}}{\rm d}x\right\vert\leq\int_\Omega\left\vert\frac{u}{d}\right\vert .\left\vert\frac{v}{d}\right\vert d^{1-\alpha}{\rm d}x\leq c\left\Vert\frac{u}{d}\right\Vert_{L^2(\Omega)}\left\Vert\frac{v}{d}\right\Vert_{L^2(\Omega)}\leq c\Vert u\Vert .\Vert v\Vert
\end{eqnarray*}
which proves the continuity. The rest follows immediately.
\begin{coro}\label{corollaire2.1}
The linear operator $P\,:\, H^{-1}(\Omega)\to H^1_0(\Omega)$ defined by $z=Ph$ is continuous.
\end{coro}
It is easy to see that solving \eqref{eq2.3} is equivalent to find fixed points of the nonlinear operator $T=i\circ P\circ F\, :\,K\to L^2(\Omega,b)$, where $i:H^1_0(\Omega)\to L^2(\Omega,b)$ is the usual Sobolev imbedding. We need a final auxiliary result from \cite{BeRo}.
\begin{lemm}[\cite{BeRo}]\label{lemme2.3}
The imbedding $H^1_0(\Omega)\to L^2(\Omega,c)$ where $c(x)=\frac{1}{d(x)^\beta}$ is compact for $\beta<2$.
\end{lemm}
\proof (of Theorem \ref{theo2.1}) The method of sub and supersolutions can be applied  since it can be shown by the usual comparison arguments that $T(K)\subset K$ with $T$ compact and monotone (in the sense that $u\leq v$ implies that $Tu \leq Tv$) and the method (see e.g., [{\sc Amann}] \cite{Am}) gives the existence of a minimal (resp. maximal) solution $\underline{u}$) (resp. $\overline{u}$) such that $u_0\leq \underline{u}\leq \overline{u}\leq u^0$.

Finally we exhibit ordered sub and super solutions satisfying the conditions in Theorem \ref{theo2.1}. As a subsolution, we try $u_0= c\phi_1$ where $-\Delta \phi_1=\lambda_1\phi_1$ in $\Omega$, $\phi=0$ on $\partial\Omega$, $\phi_1>0$, $c>0$. We have 
\begin{eqnarray*}
-\Delta u_0-\frac{1}{u_0^\alpha}=c\lambda_1\phi_1-\frac{1}{c^\alpha\phi_1^\alpha}=\frac{c^{1+\alpha}\lambda_1\phi_1^{1+\alpha}-1}{c^\alpha\phi_1^\alpha}\leq 0
\end{eqnarray*}
for $c>0$ small. As a supersolution, we pick $u^0=C\psi$, where $\psi>0$ is the unique solution to
\begin{eqnarray*}
-\Delta \psi=\frac{1}{d(x)^\alpha}\quad\mbox{in }\Omega, \quad\psi=0\quad\mbox{on }\partial\Omega.
\end{eqnarray*}
Then, we get by using that $\psi\sim d(x)$
\begin{eqnarray*}
-\Delta u^0-\frac{1}{(u^0)^\alpha}=\frac{C}{d^\alpha}-\frac{1}{(C\psi)^\alpha}=\frac{C^{\alpha+1}\psi^\alpha-c\psi^\alpha}{(C\psi)^\alpha d^\alpha}\geq 0
\end{eqnarray*}
for $C >0$ large.
\begin{rema}
Since our main goal in this paper is to show how to get existence proofs in this framework without using approximation arguments and avoiding classical linear theory, we limit ourselves to the model nonlinearity $u^{-\alpha}$; the interested reader may check that the same arguments work, with slight changes, for more general nonlinearities $f(x,u)$ "behaving like" $u^{-\alpha}$ with $0<\alpha<1$, in particular, e.g. $f(x,u)=\frac{1}{u^\alpha d(x)^\beta}$ with $\alpha+\beta<1$ and for self-adjoint uniformly elliptic differential operators.
\end{rema}

Uniqueness of the positive classical solution to $({\rm P}_0)$ was proved in \cite{CrRaTa} by using the maximum principle.  A more general uniqueness theorem which is closely related with linearized stability, was given in \cite{HeMaVe2} (see also \cite{GhRa}, \cite{HeMa} and \cite{HeMaVe}). Here we provide a very simple uniqueness proof for the solution obtained in Theorem \ref{theo2.1}. 

\begin{theo}\label{theo2.7}
Under the assumptions in Theorem \ref{theo2.1}, if $u$, $v$ are two solutions to $({\rm P}_0)$ such that $u_0\leq u,v\leq u^0$, then $u\equiv\, v$.
\end{theo}

\proof  First, we assume that $u\leq v$ in $\Omega$. Multiplying $({\rm P}_0)$ for $u$ by $v$, $({\rm P}_0)$ for $v$ by $u$ and integrating by parts on $\Omega$ with Green's formula we obtain
\begin{eqnarray*}
\displaystyle\int_{\Omega}\nabla u\cdot\nabla v{\rm d}x=\int_\Omega\frac{v}{u^\alpha}{\rm d}x=\int_\Omega\frac{u}{v^\alpha}{\rm d}x
\end{eqnarray*}
and then 
\begin{eqnarray*}
\displaystyle\int_\Omega\left(\frac{v}{u^\alpha}-\frac{u}{v^\alpha}\right){\rm d}x=\int_\Omega\frac{v^{\alpha+1}-u^{\alpha+1}}{u^\alpha v^\alpha}{\rm d}x=0.
\end{eqnarray*}
Since $v\geq u$, $u\equiv v$. Notice that all the above integrals are meaningful. Indeed, since $u,v\in K$ we have, e.g., that $\int_\Omega\frac{v}{u^\alpha}{\rm d}x\leq c\int_\Omega d(x)^{1-\alpha}{\rm d}x<\infty$.

If now $u\not\leq v$ and $v\not\leq u$, we have $u_0\leq u, v\leq u^0$. Then, $\underline{u}\leq u$, $\underline{u}\leq v$ and it follows from above that $\underline{u}=u=v$.\qed

Since this unique solution is obtained by the method of sub and supersolutions it seems natural to think that is (at least linearly) asymptotically stable. This was proved in a much more general context in \cite{HeMaVe} for solutions $u>0$ in $\Omega$ with $\frac{\partial u}{\partial n}<0$ on $\partial\Omega$ working in the space $C^1_0(\overline{\Omega})$. On the other side, the results in \cite{BeRo}, proved  working in Sobolev  spaces , are not applicable to the linearized problem we obtain for the solution $u$ above, which is actually
\begin{eqnarray}\label{eq2.5}
\left\{\begin{array}{ll}
&-\Delta w+ \alpha\frac{w}{u^{1+\alpha}}=\mu w\quad\mbox{in }\Omega,\\
&w=0\quad\mbox{on }\partial\Omega.
\end{array}\right.
\end{eqnarray}
But it is easy to give a direct proof. For this, it is clear that if such  a first eigenvalue exists in some sense, then $\mu_1>0$. It is not difficult to show the existence of an infinite sequence of eigenvalues to \eqref{eq2.5} working in $L^2(\Omega)$. Indeed, for any $z\in L^2(\Omega)$, it follows from Lemma \ref{lemme2.2} the existence  of a unique  solution to the equation \eqref{eq2.4} and it turns out that $T=i\circ P$ is a self-adjoint compact linear operator in $L^2(\Omega)$ and the classical theory gives the existence of our sequence of eigenvalues with the usual variational characterization. That $\mu_1$ is simple and has an eigenfunction $\phi_1>0$ in $\Omega$ is obtained using that, by the weak (or Stampacchia's maximum principle), $P$ is irreductible and if $z\geq 0$, $Pz>0$ and it is possible to apply the version of Krein-Rutman Theorem in the form given in [{\sc Daners-Koch-Medina}] \cite{DaKoMe} weakening the strong positivity condition for $T$ by this one (much more general results in this direction can be found in [{\sc Diaz-Hernandez-Maagli}] \cite{DiHeMa} extending most of the work in \cite{BeRo}). We have then proved:
\begin{theo}\label{theo2.8}
Problem $({\rm P}_0)$ has a unique positive solution $u_0\leq u\leq u^0$ which is linearly asymptotically stable.
\end{theo}

\begin{rema}
Linearized stability in the framework of classical solutions for much more general problems was proved in \cite{HeMaVe} working in the space $C^1_0(\overline{\Omega})$. The results in \cite{BeRo}, obtained working in weighted Sobolev spaces are not applicable here. Moreover, it is proved in \cite{HeMaVe} that linearized stability implies asymptotic stability in the sense of Lyapunov.
\end{rema}
\section{Existence in the case $1<\alpha<3$.}
We study now the same problem $({\rm P}_0)$ but for $1<\alpha<3$. If we try to apply the arguments in the preceding section, we will find some difficulties due to the fact that the embedding in Lemma \ref{lemme2.3} is not compact any more for $\beta=2$, which is precisely the critical exponent arising for $\alpha>1$.

Now we replace the assumption on the sub and supersolutions in Theorem \ref{theo2.1}  by the following:
\begin{equation}\label{eq3.1}
0<c_1 d(x)^{\frac{2}{1+\alpha}}\leq u_0\leq u^0\leq c_2 d(x)^{\frac{2}{1+\alpha}}
\end{equation}
and we define, this time for $b(x)=\frac{1}{d(x)^2}$ the set
\begin{eqnarray*}
\displaystyle K\eqdef[u_0,u^0]=\left\{u\in L^2(\Omega,b)\, :\, u_0\leq u\leq u^0\right\}.
\end{eqnarray*}
\begin{lemm}\label{lemme3.1}
There exists a constant $M>0$ such that the mapping $G: K\to H^{-1}(\Omega)$ defined by $G(w)=\frac{1}{w^\alpha} +\frac{Mw}{d(x)^2}$ is well-defined, continuous and monotone.
\end{lemm}
\proof
For the first term in $G$, we have for any $z\in H^1_0(\Omega)$ by using Hardy-Sobolev inequality
\begin{eqnarray*}
\displaystyle\left\vert<\frac{1}{w^\alpha}, z>\right\vert=\left\vert\int_{\Omega}\frac{z}{w^\alpha}{\rm d}x\right\vert=\int_{\Omega}\left\vert\frac{z}{d}\right\vert\frac{d}{w^\alpha}{\rm d}x\leq c\left\Vert\frac{z}{d}\right\Vert_{L^2(\Omega)}\int_{\Omega}d^{1-\frac{2\alpha}{1+\alpha}}{\rm d}x\leq C\Vert z\Vert
\end{eqnarray*}
since
\begin{eqnarray*}
\displaystyle\left\Vert d^{1-\frac{2\alpha}{1+\alpha}}\right\Vert_{L^2(\Omega)}=\int_{\Omega}d^{\frac{2(1-\alpha)}{1+\alpha}}{\rm d}x<+\infty\quad \mbox{(we have in fact $\frac{2(1-\alpha)}{1+\alpha}+1=\frac{3-\alpha}{1+\alpha}>0$).}
\end{eqnarray*}
For the second term of $G$, we obtain for any $z\in H^1_0(\Omega)$
\begin{eqnarray*}
\displaystyle\left\vert<\frac{w}{d^2},z>\right\vert=\left\vert\int_{\Omega}\frac{w z}{d^2}{\rm d}x\right\vert=\left\vert\int_{\Omega}\frac{w}{d}\frac{z}{d}{\rm d}x\right\vert\leq c\Vert z\Vert,
\end{eqnarray*}
again by Hardy's inequality and noticing that
\begin{eqnarray*}
\displaystyle\left\Vert\frac{w}{d}\right\Vert_{L^2(\Omega)}=\int_{\Omega}\frac{w^2}{d^2}{\rm d}x=\Vert w\Vert^2_{L^2(\Omega,b)}.
\end{eqnarray*}
We prove the continuity. For the first term we have, reasoning as above
\begin{eqnarray*}
\displaystyle\left\vert\int_{\Omega}\left(\frac{1}{w_n^\alpha}-\frac{1}{w^\alpha}\right)z{\rm d}x\right\vert=\left\vert\int_{\Omega}\frac{w^\alpha-w_n^\alpha}{w_n^\alpha w^\alpha}\left(\frac{z}{d}\right)d{\rm d}x\right\vert\leq c'_n\left\Vert\frac{z}{d}\right\Vert_{L^2(\Omega)}\leq c'_n\Vert z\Vert
\end{eqnarray*}
and using as above the mean value theorem and \eqref{eq3.1} we get
\begin{eqnarray*}
&\displaystyle c'_n=\left\Vert\frac{d(w^\alpha-w_n^\alpha)}{w_n^\alpha w^\alpha}\right\Vert_{L^2(\Omega)}=\left(\int_\Omega\frac{\alpha w(\theta)^{2(\alpha-1)}\vert w-w_n\vert^2 d^2}{\vert w_n\vert^{2\alpha}\vert w\vert^{2\alpha}}\right)^{\frac{1}{2}}
\leq c\left(\int_\Omega\frac{\vert w-w_n\vert^2}{d^4}d^2{\rm d}x\right)^{\frac{1}{2}}\leq \\
&\displaystyle c\Vert w-w_n\Vert_{L^2(\Omega,b)}
\end{eqnarray*}
giving the result.
For the second term, we write
\begin{eqnarray*}
\displaystyle\left\vert<\frac{w-w_n}{d^2},z>\right\vert=\left\vert\int_{\Omega}\frac{w-w_n}{d}\frac{z}{d}{\rm d }x\right\vert\leq c\left\Vert\frac{w-w_n}{d}\right\Vert_{L^2(\Omega)}\Vert z\Vert
\end{eqnarray*}

\begin{eqnarray*}
\displaystyle\leq c\Vert w-w_n\Vert_{L^2(\Omega,b)}\Vert z\Vert
\end{eqnarray*}
giving again the results. On the other side, the existence of a constant $M$ is proved in the same way.
\qed

We still have the 
\begin{lemm}\label{lemme3.2}
If $1<\alpha<3$, for any $h\in H^{-1}(\Omega)$, there exists a unique solution $z\in H^1_0(\Omega)$ to the linear problem
\begin{eqnarray}\label{eq3.2}
\displaystyle\left\{\begin{array}{ll}
&-\Delta u+ M\frac{u}{d(x)^2}=h\\
& u=0\mbox{ on }\partial\Omega.
\end{array}\right.
\end{eqnarray}
\end{lemm}
\proof
It is very similar to the case in Lemma \ref{lemme2.2}, using again Hardy inequality. 

But now we cannot argue as in the proof of Theorem \ref{theo2.1}, the reason is that the embedding in Lemma \ref{lemme2.3} is not compact any more if $\beta=2$. This fact also raises problems when studying linear singular eigenvalue problems in \cite{BeRo}, see also \cite{DiHeMa}. This difficulty may be circumvented as follows. From Lemmas \ref{lemme3.1} and \ref{lemme3.2}, we can construct the following iterative scheme starting from the surpersolution $u^0$:
\begin{eqnarray*}
\left\{\begin{array}{ll}
&-\Delta u^n+\frac{M u^n}{d^2(x)}=\frac{1}{(u^{n-1})^\alpha}+\frac{M u^{n-1}}{d^2(x)}\quad\mbox{in }\Omega,\\
& u=0\quad\mbox{on }\partial\Omega,
\end{array}\right.
\end{eqnarray*}
and a similar one starting this time from the subsolution $u_0$. By using the usual comparison principle arguments we get two monotone sequences satisfying:
\begin{eqnarray*}
u_0\leq u_1\leq\cdot\cdot\cdot\leq u_n\leq\cdot\cdot\cdot\leq u^n\leq\cdot\cdot\cdot\leq
 u^1\leq u^0
\end{eqnarray*}
and it follows that there are subsequences $u_n$ and $u^n$ such that $u_n\to\underline{u}$ and $u^n\to \overline{u}$ pointwise. By exploiting the regularity for the above linear problem and the estimates in Lemma \ref{lemme3.1} we obtain the uniform estimate:
\begin{eqnarray*}
\displaystyle\Vert u^n\Vert_{H^1_0(\Omega)}\leq c\left\Vert\frac{1}{(u^{n-1})^\alpha}+\frac{Mu^{n-1}}{d^2(x)}\right\Vert_{H^{-1}(\Omega)}\leq c
\end{eqnarray*}
where $c$ is a constant independent of $n$. Thus there exists again subsequences $u_n$ and $u^n$ such that $u_n\to u_*$ and $u^n\to u^*$ weakly in $H^1_0(\Omega)$ and then strongly in $L^2(\Omega)$. Obviously, $u_*=\underline{u}$ and $u^*=\overline{u}$. 

Next we should pass to the limit in the above equation \eqref{eq3.2}. The weak formulation is
\begin{eqnarray*}
\displaystyle\int_{\Omega}\nabla u^n\nabla\phi{\rm d}x+M\int_{\Omega} \frac{u^n}{d^2(x)}\phi{\rm d}x=\int_{\Omega}\frac{\phi}{(u^{n-1})^\alpha}{\rm d}x+M\int_{\Omega} \frac{u^{n-1}}{d^2(x)}\phi{\rm d}x
\end{eqnarray*}
for any $\phi\in H^1_0(\Omega)$. The first term on the left-hand side of the above expression converges clealy to $\int_{\Omega}\nabla\overline{u}\nabla\phi$. Concerning the first term on the right-hand side we have, by using the dominate convergence thorem, that there is pointwise convergence to $\frac{\phi}{(\overline{u})^{\alpha}}$. Moreover,
\begin{eqnarray*}
\displaystyle\left\vert\int_{\Omega}\frac{\phi}{(u^{n-1})^{\alpha}}{\rm d}x\right\vert=\left\vert\int_{\Omega}\frac{\phi}{d}\frac{d}{(u^{n-1})^\alpha}{\rm d}x\right\vert\leq c\left\Vert\frac{\phi}{d}\right\Vert_{L^2(\Omega)}\left\Vert\frac{d}{(u^{n-1})^\alpha}\right\Vert_{L^2(\Omega)}
\end{eqnarray*}
where $c$ does not  depend on $n$. We have 
\begin{eqnarray*}
\left\Vert\frac{d}{(u^{n-1})^\alpha}\right\Vert^2_{L^2(\Omega)}=\int_{\Omega}\frac{d^2}{(u^{n-1})^{2\alpha}}{\rm d}x\leq c\int_{\Omega}d^{2-\frac{4\alpha}{1+\alpha}}<+\infty
\end{eqnarray*}
since $1+\frac{2(1-\alpha)}{1+\alpha}=\frac{3-\alpha}{1+\alpha}>0$. For the second terms on both sides we have
\begin{eqnarray*}
\left\vert\int_{\Omega}\frac{u^n\phi}{d^2}{\rm d}x\right\vert=\int_{\Omega}\left\vert\frac{\phi}{d}\right\vert\left\vert\frac{u^n}{d}\right\vert{\rm d}x\leq\left\Vert\frac{\phi}{d}\right\Vert_{L^2(\Omega)}\left\Vert\frac{u^n}{d}\right\Vert_{L^2(\Omega)}
\end{eqnarray*}
and 
\begin{eqnarray*}
\left\Vert\frac{u^n}{d}\right\Vert^2_{L^2(\Omega)}=\int_{\Omega}\left(\frac{u^n}{d}\right)^2{\rm d}x\leq c\int_{\Omega}d^{\frac{2(1-\alpha)}{1+\alpha}}{\rm d}x<\infty
\end{eqnarray*}
as above. 

It only remains to find ordered sub and supersolutions for the problem. It seems natural to look for functions of the form $c\phi_1^t$ with $t=\frac{2}{1+\alpha}<1$. For the subsolution $u_0$, we obtain 
\begin{eqnarray*}
-\Delta (\phi_1^t)=\phi_1^{t-2}\left(t(1-t)\vert\nabla\phi_1\vert^2+\lambda_1 t\phi_1^2\right)=\lambda_1 t\phi_1^t+ t(1-t)\phi_1^{t-2}\vert\nabla\phi_1\vert^2.
\end{eqnarray*}
Hence we get
\begin{eqnarray*}
&-\Delta u^0-\frac{1}{(u^0)^\alpha}=ct(t-1)\phi_1^{t-2}\vert\phi_1\vert^2+c\lambda_1 t\phi_1^t-\frac{1}{c^\alpha\phi_1^{\alpha t}}=\\
 &\frac{c t(t-1)\vert \nabla \phi_1\vert^2}{\phi_1^{\frac{2\alpha}{1+\alpha}}}+\lambda_1 ct\phi_1^t-\frac{1}{c^\alpha\phi_1^{\frac{2\alpha}{1+\alpha}}}\leq 0
\end{eqnarray*}
using that $t-2=-\frac{2\alpha}{1+\alpha}$, and this is equivalent to
\begin{eqnarray*}
t(1-t)\vert\nabla\phi_1\vert^2+\lambda_1 t\phi_1^{t+\frac{2\alpha}{1+\alpha}}\leq \frac{1}{c^{\alpha+1}}.
\end{eqnarray*}
Hence it is enough to have 
\begin{eqnarray*}
t(1-t)\vert\nabla\phi_1\vert^2+\lambda_1 t\leq \frac{1}{c^{\alpha +1}}
\end{eqnarray*}
which is satisfied for $c>0$ small.

Reasoning in a similar way for the supersolution $u^0=C\phi_1^t$, we infer that
\begin{eqnarray*}
t(1-t)\vert\nabla \phi_1\vert^2+\lambda_1 t\phi_1^{t+\frac{2\alpha}{1+\alpha}}\geq \frac{1}{C^{1+\alpha}}.
\end{eqnarray*}
We know that $\vert\nabla\phi_1\vert\geq \delta_1>0$ in $\Omega_\varepsilon\eqdef\left\{x\in\Omega\vert d(x)\leq \varepsilon\right\}$ for some $\varepsilon>0$. Then,
\begin{eqnarray*}
t(1-t)\vert\nabla\phi_1\vert^2\geq t(1-t)\delta_1^2\geq \frac{1}{C^{1+\alpha}}
\end{eqnarray*}
on $\Omega_\varepsilon$ for $C>C_1>0$ large enough. On $\Omega\backslash\Omega_\varepsilon$, we have that $\phi_1\geq \delta_2$ for some $\delta_2>0$ and it is enough to have
\begin{eqnarray*}
\lambda_1 t\delta_2^{t+\frac{2\alpha}{1+\alpha}}\geq \frac{1}{C^{1+\alpha}}
\end{eqnarray*}
which is satisfied for $C>C_2$ for some $C_2>0$ large enough. Finally we pick $C>\max(C_1,C_2)$.\qed

We have then proved 
\begin{theo}\label{theo3.3}
Assume that there exists a subsolution $u_0$ (resp. a supersolution $u^0$) satisfying \eqref{eq3.1}. Then there exists a minimal solution $\underline{u}$ (resp. a maximal solution $\overline{u}$) such that
\begin{eqnarray*}
u_0\leq\underline{u}\leq\overline{u}\leq u^0.
\end{eqnarray*}
The uniqueness and linearized stability are obtained in this case as well. Since proofs are very similar, we only point out the differences.
\end{theo}

\begin{theo}
Under the assumptions of Theorem \ref{theo3.3}, there is a unique solution in the interval $[u_0, u^0]$ which is linearly asymptotically stable.
\end{theo}
\proof For uniqueness the same arguments in Theorem \ref{theo2.7} work here as well. We only show that all integrals are meaningful. We have, e.g., that
\begin{eqnarray*}
\int_\Omega\frac{v}{u^\alpha}{\rm d}x\leq c\int_\Omega d(x)^{1-\alpha}{\rm d}x\leq c\int_\Omega d(x)^{\frac{2(1-\alpha)}{1+\alpha}}{\rm d}x<\infty
\end{eqnarray*}
since $\frac{2(1-\alpha)}{1+\alpha}+1=\frac{3-\alpha}{1+\alpha}>0$.

For linearized stability it is enough to check that all the arguments at the end of Section 2 still work taking into account that $u^{1+\alpha}$ "behaves like" $d(x)^2$ and using again Hardy's inequality.
\section{Regularity of weak solutions}
We deal now with the following elliptic problem  $({\rm P_1})$:
\begin{equation*}
(\rm P_1)\qquad 
\left  \{ \begin{array}{l}
-\Delta u = \dfrac{1}{d^\beta u^{\alpha}}\quad \text{ in } \Omega \\
u = 0\  \text{ on } \partial\Omega,\ u>0 \text{ on } \Omega,
\end{array}
\right .
\end{equation*}
where $\Omega$ is an open bounded domain with smooth boundary in $\mathbb R^{N}$, $\alpha\in \mathbb R$, $\ 0\leq \beta < 2$. We prove the following regularity result for solutions to $({\rm P_1})$:
\begin{theo}\label{regularity}
Let $\alpha+\beta>1$. Then the unique positive  solution $u\in C^2(\Omega)\cap C^0(\overline{\Omega})$ to Problem $({\rm P_1})$ satisfies 
\begin{equation}\label{eee}
u \in W^{1,q}_0(\Omega)\quad \text{ for all } 1<q< \bar{q}_{\alpha,\beta}=\dfrac{1+\alpha}{\alpha + \beta -1}.
\end{equation}
Furthermore, the restriction given by $\bar{q}_{\alpha,\beta}$ is sharp.
\end{theo}
\begin{rema}
\begin{itemize}
\item[(i)] The uniqueness of the positive solution to $({\rm P_1})$ follows from the classical strong maximum principle.
\item[(ii)] The existence of $u$ can be obtained by the same approximation procedure as in \cite{CrRaTa} and $u\in \mathcal C^+_{\phi_{\alpha,\beta}} (\overline{\Omega})$ where 
\begin{equation}
\mathcal C^+_{\phi_{\alpha,\beta}} (\overline{\Omega})= \{v \in  C(\overline{\Omega})\ | \quad \exists\  c_1,\ c_2>0\   : \ c_1 \phi_{\alpha,\beta} \leq v \leq c_2 \phi_{\alpha,\beta}\ \text{  a.e. in } \Omega\}
\end{equation}
with $\phi_{\alpha,\beta}\eqdef\phi_1^{\frac{2-\beta}{1+\alpha}}$ when $\alpha+\beta>1$. Existence of very weak solutions was proved also in \cite{DiHeRa}.
\item[(iii)] Theorem \ref{regularity} still holds when $\frac{1}{d(x)^\beta}$ is replaced  by  a more general weight $K_0(x)$ behaving like $\frac{1}{d(x)^\beta}$ near $\partial\Omega$.
\item[(iv)] If $\alpha+\beta<1$, we know that $u\in C^{1,\mu}(\overline{\Omega})$ for some $\mu\in (0,1)$ (see \cite{GuLi}). Theorem \ref{regularity} complements to some extent results in \cite{GuLi}.
\end{itemize}
\end{rema}
To prove Theorem \ref{regularity}, we use the following result concerning interior regularity for linear elliptic problems (see {\sc Bers-John-Schechter} \cite[Theorem 4, Chapter 5]{BeJoSc} or Lemma 1.5 in \cite{CrRaTa}):
\begin{lemm} \label{Locale-regularity} Let $D_0$ and $D$ be opem bounded domains in $\mathbb R^N$ with $\overline{D}_0 \subset D$. Assume that $L$ is a second order uniformly elliptic operator with coefficients in $\mathcal{C}(\overline{D}) $ and let $q>N$. Then there exists a positive constant $K=K(N,q,\delta(D), d(D_0,\partial D),L)$ such that for any $w \in W^{2,q}_0(D)$
\begin{equation}
\|w\|_{W^{2,q}(D_0)} \leq K\left (\|Lw\|_{L^q(D)}  + \|w\|_{L^q(D)}\right ).
\end{equation}
In particular we have the estimate
\begin{equation}\label{Bers-local-estimate}
\|w\|_{W^{2,q}(D_0)} \leq K\left (\|Lw\|_{L^\infty(D)}  + \|w\|_{L^\infty(D)}\right ).
\end{equation}
\end{lemm}
We have the following auxiliary result.
\begin{lemm}
There exists a constant $K_1>0$ such that if $r \in (0,1]$, $x_0 \in \Omega$, $B_{2r}(x_0) = \{x\in \mathbb R^N |  |x-x_0|<2r\} \subset \Omega$ and $v \in W^{2,q}(B_{2r}(x_0))$ where $q>N$, then 
\begin{equation} \label{local-bound}
 |\nabla v(x)| \leq K_1\left (r \|\Delta v\|_{L^\infty(B_{2r}(x_0))} + \dfrac{1}{r} \|v\|_{L^\infty(B_{2r}(x_0))}\right)
\end{equation}
for all $ x \in B_{r}(x_0) $. (Here $\|\Delta v\|_{L^\infty(B_{2r}(x_0))} =\infty$  is included).
\end{lemm}

\begin{proof}
Let $x_0 \in \Omega$, and let $r : 0< 2r < d(x_0)$ (then $B_{2r}(x_0) \subset \Omega$). We make the change of variable $x_0 + ry=x$ and define $w(y) = v(x)$, for $y \in \overline {B_{2}(0)}$. Then we have 
\begin{equation} \label{chang-var}
\nabla w(y) = r\nabla v(x),\quad \Delta w(y) =r^2 \Delta v(x) \text{    for  } |y|\leq 2
\end{equation}
and by  using \eqref{Bers-local-estimate}, we obtain
\begin{equation}\label{libermann}
|\nabla w(y)|  \leq K_1 \left ( \|\Delta w\|_{L^\infty(B_{2}(0))} + \|w\|_{L^\infty (B_{2}(0))} \right ), \text{ for all  } y \in B_{1}(0)
\end{equation}
for some constant $K_1>0$ independent of $r$ and $x_0$. Hence, the local estimate \eqref{local-bound} follows from \eqref{chang-var} and \eqref{libermann}.
\end{proof}
\begin{lemm}\label{esti-grad}
Assume hypothesis in Theorem \ref{regularity}. Then, any weak solution $u$ to $({\rm P_1})$ in $\mathcal  C^+_{\phi_{\alpha,\beta}} (\Omega)$  satisfies
\begin{equation}\label{gradient-estimate}
|\nabla u(x)|\leq c d(x)^{\frac{1-\alpha-\beta}{1+\alpha}} \text{ for all }x\in \Omega.
\end{equation}
\end{lemm}
\begin{proof}
Let $x \in \Omega$ and set $r= \dfrac{d(x)}{3}$, $v=u$, (so $\Delta v  = \Delta u =d^{-\beta}u^{-\alpha}$) and we take $x_0 = x$. Let us note that 
$$B_{2r}(x)  \subset A =\{z \in \Omega : \dfrac{d(x)}{3}\leq d(z) \leq \dfrac{5}{3} d(x)\}\subset \Omega.$$
Using \eqref{local-bound}, we obtain
\begin{equation} \label{aaa}
|\nabla u(x)| \leq K_2\left (d(x) \|d^{-\beta}u^{-\alpha}\|_{L^\infty(A)} + \dfrac{1}{d(x)} \|u\|_{L^\infty(A)}\right )
\end{equation}
where $K_2 =3K_1$. Since $u\in\mathcal  C^+_{\phi_{\alpha,\beta}} (\Omega)$,
we have that
$a d(x)^{\frac{2-\beta}{1+\alpha}}\leq u(x)\leq b d(x)^{\frac{2-\beta}{1+\alpha}}$ for some $a,b>0$. Then,
\begin{equation}\label{bbb}
d(x)\|d^{-\beta}u^{-\alpha}\|_{L^\infty(A)} \leq a d(x) \| d^{-\beta}d^{\frac{-(2-\beta)\alpha}{1+\alpha}}\|_{L^\infty(A)} = a' d(x)^{\frac{1-\alpha-\beta}{1+\alpha}}
\end{equation}
and 
\begin{equation}\label{ccc}
 \frac{1}{d(x)}\|u\|_{L^\infty(A)} \leq b d(x) \|d^{\frac{2-\beta}{1+\alpha}}\|_{L^\infty(A)}  = b' d(x)^{\frac{1-\alpha-\beta}{1+\alpha}}.
\end{equation} 
Then the estimate \eqref{gradient-estimate} follows from \eqref{aaa}, \eqref{bbb} and \eqref{ccc}.
\end{proof}
\textbf{Proof of Theorem \ref{regularity}}
Indeed, reasoning as in [{\sc Lazer-Mc Kenna}] \cite{LaKe} by rectifying the boundary using the smoothness of $\partial\Omega$ and a partition of the unity, the problem of finding $q>1$ such that $\nabla u \in L^q(\Omega)$ is reduced from Lemma \ref{esti-grad} to 
\begin{eqnarray*}
\int_{\Omega}d(x)^{\frac{q(1-\alpha-\beta)}{1+\alpha}}<\infty,
\end{eqnarray*}
 that is $\frac{q(1-\alpha-\beta)}{1+\alpha}+1>0$, which gives the result.\qed

\end{document}